\documentclass[a4paper,12pt,reqno]{amsart}

\usepackage{caption} 
\usepackage{geometry}
\usepackage{setspace}
\usepackage{indentfirst}
\usepackage{hyperref}
\usepackage{amsmath,amssymb,amsfonts,amsthm,amsopn,
	bbm}
\usepackage[mathscr]{eucal}
\makeatletter
\def\specialsection{\@startsection{section}{1}%
	\z@{\linespacing\@plus\linespacing}{.5\linespacing}%
{\normalfont}}
\def\section{\@startsection{section}{1}%
\z@{.7\linespacing\@plus\linespacing}{.5\linespacing}%
{\normalfont\scshape\bfseries}}
\makeatother
\theoremstyle{definition}
\newtheorem{theorem}{Theorem}[section]
\newtheorem{lemma}[theorem]{Lemma}
\newtheorem{corollary}[theorem]{Corollary}
\newtheorem{definition}[theorem]{Definition}

\newtheorem*{lemma*}{Lemma}
\newtheorem*{theorem*}{Theorem}

\newtheorem{remark}[theorem]{Remark}
\newtheorem*{remark*}{Remark}

\newcommand{\D}{\Deg_{\max}}
\newcommand{\IR}{{\mathbb{R}}}

\newcommand{\eChar}{\begin{enumerate}[(i)]}
	\newcommand{\eCharR}{\begin{enumerate}[(a)]}
		\newcommand{\eBr}{\begin{enumerate}[(1)]}
			
			\newcommand{\Deg}{\operatorname{Deg}}
			
			\newcommand{\diam}{\operatorname{diam}}
			\newcommand{\eps}{\varepsilon}
			
			\newcommand{\Abstract}

    \numberwithin{equation}{section}
    \thanks{Research of the authors is partially supported by the National NSF grant of China (no. 11801274), the Visiting Scholar Program from the Department of Mathematical Sciences of Tsinghua University, and the Open Project from Yunnan Normal University (no. YNNUMA2403). YCH thanks in particular Xueping Huang and Shiping Liu for helpful communications.}
\begin{document}
	\title{Refined Diameter Bounds under Curvature Dimension Conditions}
	\date{}
	\author{Yi C. Huang}
	\address{School of Mathematical Sciences, Nanjing Normal University, Nanjing 210023, People's Republic of China}
	\email{Yi.Huang.Analysis@gmail.com}
	\urladdr{https://orcid.org/0000-0002-1297-7674}
	\author{Ze Yang}
	\address{School of Mathematical Sciences, Nanjing Normal University, Nanjing 210023, People's Republic of China}
	\email{506017957@qq.com}
	\maketitle
\begin{abstract}
 In this article we derive an explicit diameter bound for graphs satisfying the so-called curvature dimension conditions $CD(K,n)$. This refines a recent result due to Liu, M\"unch and Peyerimhoff when the dimension $n$ is finite.
\end{abstract}
\section{Introduction}
First, we prepare some notions about discrete geometric analysis (for further materials in this rapidly growing area, one can refer to the monographs by Bakry-Gentil-Ledoux \cite{bakry2014analysis} and Keller-Lenz-Wojciechowski \cite{keller2021graphs}).
\begin{definition}[Weighted graphs and nondegeneracy]
A triple $G=(V,w,m)$ is called a \textit{weighted (and measured) graph} if $V$ is a countable set, if $w:V^2 \to [0,\infty)$ is symmetric, i.e., 
$$w(x,y)= w(y,x), \quad\forall\, x,y\in V,$$ 
and zero on the diagonal, i.e., $$w(x,x)=0, \quad\forall\, x\in V,$$ 
and if $m:V \to (0,\infty)$. We call $V$ the \textit{vertex set}, $w$ the \textit{edge weight} and $m$ the \textit{vertex measure}. If $w(x,y)>0$, we write $x\sim y.$ The measure $m$ on $V$ is called  \textit{non-degenerate}, if $ \inf_{x \in V} m(x) >0$.
\end{definition}
\begin{definition}[Locally finite graphs, vertex degree, and connectedness] 
A triple $G=(V,w,m)$ is called a \textit{locally finite graph} if $$\deg(x) :=\sum_{y\in V} w(x,y)<\infty, \quad\forall\, x\in V.$$ 
Moreover, it is called a \textit{connected graph} if for all $x, y\in V$ there exists a sequence $\{x_i\}_{i=0}^n\subset V$ such that $x=x_0\sim x_1\sim \cdots \sim x_n=y$.
\end{definition}
\begin{definition} [Laplacians on graphs]
We define the  \textit{graph Laplacian} $\Delta$ as
$$\Delta f(x) := \frac 1 {m(x)} \sum_{y\in V} w(x,y)(f(y) - f(x)), \quad x\in V.$$
Here, $G=(V,w,m)$ is a locally finite graph and $f:V \to \mathbb{R}.$
\end{definition}
\begin{definition} [Bakry-\'Emery curvature dimension conditions]
	The  \textit{Bakry-\'Emery operators} $\Gamma$ and $\Gamma_2$ are defined via
	$$
	2\Gamma(f,g) := \Delta (fg) - f\Delta g - g\Delta f
	$$
	and
	$$
	2\Gamma_2(f,g) := \Delta \Gamma(f,g) - \Gamma(f, \Delta g) - \Gamma(g,\Delta f).
	$$
	We write for simplicity $\Gamma f= \Gamma(f,f)$ and $\Gamma_2(f)=\Gamma_2(f,f)$.
	
	A graph $G$ is said to satisfy the  \textit{curvature dimension condition} $CD(K,n)$ for some $K\in \mathbb{R}$ and $n\in (0,\infty]$ if for all $f$,
	$$
	\Gamma_2(f) \geq \frac 1 n (\Delta f)^2 + K \Gamma f.
	$$
\end{definition}
\begin{definition}[Combinatorial distance]\label{def:combinatorial distance}
	Let $G=(V,w,m)$ be a connected graph.
	We define the \textit{combinatorial distance} $d:V^2 \to [0,\infty)$ via
	\begin{align*}
		d(x,y) := \min \{n : \mbox{ there exist } x=x_0,x_1,\cdots,x_n=y \mbox{ s.t. } x_i\sim  x_{i-1}\,\,\forall\, i=1\ldots n\}
	\end{align*}
	and the  \textit{combinatorial diameter} as
	$\diam_d(G) := \sup_{x,y \in V} d(x,y)$.
\end{definition}

\begin{definition}[Resistance distance]\label{def:resistance distance}
	Let $G=(V,w,m)$ be a locally finite graph.
	We define the  \textit{resistance distance} $\rho:V^2 \to [0,\infty)$ via
	$$
	\rho(x,y) := \sup \{|f(y) - f(x)| : \left\| \Gamma f \right\|_\infty \leq 1\}
	$$
where $\left\| \Gamma f \right\|_\infty=\sup_{x \in V} (\Gamma f)(x)$,	and the  \textit{resistance diameter} as
	$$ \diam_\rho(G) := \sup_{x,y \in V} \rho(x,y).$$
\end{definition}
By \cite[Lemma 1.4]{liu2018bakry}, we have the following relationship between the combinatorial and resistance distances. We also define $$\Deg(x) := \frac{\deg(x)}{m(x)}\text{ \,\,and \,\,}\D := \sup_{x\in V} \Deg(x).$$
\begin{lemma}[Liu, M\"unch and Peyerimhoff]\label{lem 1.7}
	Let $G=(V,w,m)$ be a locally finite connected graph with $\D < \infty$. Then for all $x_0,y_0 \in V$, 
	$$
	d(x_0,y_0) \leq \sqrt{\frac{\D}{2}}\rho(x_0,y_0).
	$$
\end{lemma}
The following explicit resistance distance bound is our main result.
\begin{theorem}\label{them 1.8}
	Let $G=(V,w,m)$ be a locally finite, connected, complete graph with non-degenerate vertex measure satisfying $CD(K,n)$ for $K>0$, $n<\infty$ and $\D<\infty$.
	Then for all $x_0,y_0 \in V$,
	\begin{equation} \label{e:main}
	\rho(x_0,y_0) \leq  \sqrt{\frac{n}{K}}\left( \arcsin \frac{1}{\sqrt{\frac{Kn}{2\Deg\left( x_0 \right)}+1}}+\arcsin \frac{1}{\sqrt{\frac{Kn}{2\Deg\left( y_0 \right)}+1}} \right).
	\end{equation}
\end{theorem}

\begin{remark}
	The completeness assumption is needed when using certain functional consequences of the curvature dimension conditions, see Lemma \ref{lem 2.1} below. For the definition of completeness of graphs, see \cite[Sections 1 and 2.1]{hua2017stochastic} or \cite[Definition 2.9, Definition 2.13]{gong2015properties}.
	Our resistance distance bound refines a recent result due to Liu, M\"unch and Peyerimhoff in the sense the two \text{arcsin}-s in \eqref{e:main} sum to $\pi$ as $\min\{\Deg\left( x_0 \right), \Deg\left( y_0 \right)\} \rightarrow \infty$.	\eqref{e:main} also holds when $n=\infty$, see \cite[Theorem 2.1]{liu2018bakry}.
\end{remark}

\begin{remark}
For diameter bounds via a different curvature, see Steinerberger \cite{Ste}.
\end{remark}

\begin{corollary}[Diameter bounds under $CD(K,n)$]\label{coro1.10}
	Let $G=(V,w,m)$ be a locally finite, connected, complete graph with non-degenerate vertex measure satisfying $CD(K,n)$ for $n<\infty$ and $\D < \infty$.
	Then
	$$
	\diam_d(G) \leq \sqrt{\frac{2n\D}{K}} \arcsin \frac{1}{\sqrt{\frac{Kn}{2\D}+1}}.
	$$
\end{corollary}
\section{Proofs of Theorem \ref{them 1.8} and Corollary \ref{coro1.10}}
By \cite[Lemma 2.3]{liu2018bakry}, we have the following consequence under $CD(K,n)$.
\begin{lemma}[Liu, M\"unch and Peyerimhoff]\label{lem 2.1}
	Let $G=(V,w,m)$ be a complete graph with non-degenerate vertex measure.
	Suppose $G$ satisfies $CD(K,n)$. Then for all bounded $f:V \to \IR$ with bounded $\Gamma f$,
	$$\label{eqn:CD(K,n) semigroup characterization}
	\Gamma P_t f \leq e^{-2Kt} P_t \Gamma f - \frac{1-e^{-2Kt}}{Kn}(\Delta P_t f)^2.
	$$
	Here, $P_t$ denotes the heat semigroup associated to $\Delta$.
\end{lemma}
\begin{proof}[Proof of Theorem \ref{them 1.8}]
	We follow the streamlined arguments for \cite[Theorem 2.4]{liu2018bakry}. By Lemma \ref{lem 2.1}, we see that $CD(K,n)$ is equivalent to
	$$\label{eq:semigrpchar}
	\Gamma P_t f \leq e^{-2Kt} P_t \Gamma f - \frac{1-e^{-2Kt}}{Kn}(\Delta P_t f)^2	$$
	for all bounded function $f:V \to \mathbb{R}$. We fix $x_0,y_0 \in V$ and $\varepsilon  > 0 $. Then by the definition of $\rho$, there is a function $f:V\to \mathbb{R}$ s.t.
	$f(y_0)-f(x_0) > \rho(x_0,y_0) - \varepsilon  $ and $\Gamma f \leq 1$. W.l.o.g., we assume that $f$ is bounded. When $\D<\infty$, we obtain
	\begin{align*}
		\Gamma P_t f(x)
		&=\frac{\sum_{y\in V}{w(x,y) (P_t f(y) -P_t f(x)) ^2}}{2m(x)}\nonumber\\
		&\geq \frac{\left( \sum_{y\in V}{\frac{w(x,y)(P_t f(y) -P_t f(x))}{m(x)}}\right) ^2}{2\frac{\sum_{y\in V}{w(x,y)}}{m(x)}}=\frac{\left(\Delta P_t f(x)\right) ^2}{2\Deg(x)}, \quad x\in V\nonumber.
	\end{align*}
	Combined with $CD(K,n)$ condition, we obtain
	$$\left(\frac{1}{2\Deg(x)}+\frac{1-e^{-2Kt}}{Kn} \right) (\Delta P_tf(x))^2 \leq e^{-2Kt}P_t(\Gamma f)(x), \quad x\in V.$$
	\\
	Taking square root and applying $\Gamma f\leqslant 1 $, we have
	\begin{align*}
		| \partial _tP_tf(x) |
		&=| \Delta P_tf(x) |\\
		&\leq \frac{e^{-Kt}}{\sqrt{\frac{1}{2\Deg(x)}+\frac{1-e^{-2Kt}}{Kn}}}=\sqrt{Kn}\frac{e^{-Kt}}{\sqrt{\left(\frac{Kn}{2\Deg(x)}+1\right) -e^{-2Kt}}}, \quad x\in V.
	\end{align*}
	\\
	When $T>0$, we have
	\\
	\begin{align*}
		|P_T f(x) -f(x)|
		&\leq \int_0^{T}{\left| \partial _tP_tf(x) \right| dt}\\
		&\leq \int_0^{\infty}{\left| \partial _tP_tf(x) \right| dt}\\
		&\leq \int_0^{\infty}{\sqrt{Kn}\frac{e^{-Kt}}{\sqrt{\left( \frac{Kn}{2\Deg(x)}+1\right) -e^{-2Kt}}}dt}\nonumber\\
		&=\sqrt{\frac{n}{K}}\arcsin \frac{1}{\sqrt{ \frac{Kn}{2\Deg(x)}+1}}\nonumber, \quad x\in V.
	\end{align*}
	\\
	Hence combined with the triangle inequality we obtain
	\begin{align*}
		\rho (x_0,y_0)-\varepsilon 
		&\leq |f(x_0)-f(y_0)|\nonumber\\
		&\leq \left| P_Tf(x_0)-f(x_0) \right|+\left| P_Tf(x_0)-P_Tf(y_0) \right|+\left| P_Tf(y_0)-f(y_0) \right|\nonumber\\
		&\leq \sqrt{\frac{n}{K}}\left( \arcsin \frac{1}{\sqrt{\frac{Kn}{2\Deg\left( x_0 \right)}+1}}+\arcsin \frac{1}{\sqrt{\frac{Kn}{2\Deg\left( y_0 \right)}+1}} \right)\\
		&\qquad\qquad\qquad +\left| P_T f(x_0)-P_T f(y_0) \right|\nonumber\\
		&\xrightarrow{t\rightarrow \infty}\sqrt{\frac{n}{K}}\left( \arcsin \frac{1}{\sqrt{\frac{Kn}{2\Deg\left( x_0 \right)}+1}}+\arcsin \frac{1}{\sqrt{\frac{Kn}{2\Deg\left( y_0 \right)}+1}} \right) \nonumber.
	\end{align*}
	Due to $CD(K,n)$ which implies 
	$\| \Gamma P_T f \|_\infty {\rightarrow} 0$ as $T\to \infty$, and since $G$ is connected,
	$$\left|P_T f(x_0) - P_T f(y_0) \right| {\rightarrow} 0\text{\,\, as \,\,}T\to \infty.$$
	Taking the limit $\eps \to 0$ finishes the proof.
\end{proof}
\begin{proof}[Proof of Corollary \ref{coro1.10}]
By Lemma \ref{lem 1.7}, we obtain
\begin{align*}
	d(x_0,y_0) 
	&\leq \sqrt{\frac{\D}{2}} \rho(x_0,y_0) \\
	&\leq  \sqrt{\frac{n\D}{2K}}\left( \arcsin \frac{1}{\sqrt{\frac{Kn}{2\Deg\left( x_0 \right)}+1}}+\arcsin \frac{1}{\sqrt{\frac{Kn}{2\Deg\left( y_0 \right)}+1}} \right)\nonumber\\
	&\leq \sqrt{\frac{2\mathrm{\D}n}{K}}\arcsin \frac{1}{\sqrt{\frac{Kn}{2\D}+1}}\nonumber.
\end{align*}
Taking the supremum over $x_0, y_0\in V$ finishes the proof.
\end{proof}


\section*{\textbf{Compliance with ethical standards}}


\textbf{Conflict of interest} The authors have no known competing financial interests
or personal relationships that could have appeared to influence this reported work.


\textbf{Availability of data and material} Not applicable.

\bibliographystyle{plain}

\begin{thebibliography}{1}

\bibitem{bakry2014analysis}
Dominique Bakry, Ivan Gentil, and Michel Ledoux.
\newblock {\em Analysis and Geometry of Markov Diffusion Operators}, volume 348
  of {\em Grundlehren der mathematischen Wissenschaften}.
\newblock 2014.

\bibitem{gong2015properties}
Cao Gong and Yong Lin.
\newblock Equivalent properties for {CD} inequalities on graphs with unbounded {L}aplacians.
\newblock {\em  Chinese Annals of Mathematics, Series B}, 38:1059--1070, 2017.

\bibitem{hua2017stochastic}
Bobo Hua and Yong Lin.
\newblock Stochastic completeness for graphs with curvature dimension
  conditions.
\newblock {\em Advances in Mathematics}, 306:279--302, 2017.

\bibitem{keller2021graphs}
Matthias Keller, Daniel Lenz, and Radoslaw~K. Wojciechowski.
\newblock {\em Graphs and Discrete Dirichlet Spaces}, volume 358 of {\em
  Grundlehren der mathematischen Wissenschaften}.
\newblock 2021.

\bibitem{liu2018bakry}
Shiping Liu, Florentin M{\"u}nch, and Norbert Peyerimhoff.
\newblock Bakry--\'{E}mery curvature and diameter bounds on graphs.
\newblock {\em Calculus of Variations and {P}artial {D}ifferential
  {E}quations}, 57:1--9, 2018.

\bibitem{Ste}
Stefan Steinerberger.
\newblock Curvature on graphs via equilibrium measures.
\newblock {\em Journal of Graph Theory}, 103:415--436, 2023. 
  
  \end{thebibliography}

\end{document}